\documentclass[12pt, a4paper]{amsart}

\usepackage{amsmath, amsthm, amssymb}
\usepackage{graphicx}

\usepackage{enumerate}
\usepackage[top=2.5cm, bottom=2.5cm, left=2.5cm, right=2.5cm]{geometry}

\newtheorem{theorem}{\bf Theorem}
\newtheorem{corollary}{\bf Corollary}

\theoremstyle{definition}
\newtheorem{remark}{Remark}{\rm}
{\rm}

\usepackage{color}

\usepackage[T1]{fontenc}

\usepackage{mathrsfs}
\usepackage{soul}

\newcommand{\Fix}{\ensuremath{\operatorname{Fix}}}
\newcommand{\sFix}{\ensuremath{\operatorname{{\bold{Fix}}}}}

\begin{document}

\title{On the local convergence of the\\ Douglas--Rachford algorithm}
\date{January 23, 2014}
\author{H.H. Bauschke$^\dag$, D. Noll$^*$}
\thanks{$^\dag$ University of British Columbia, Kelowna, Canada}
\thanks{$^*$ Institut de Math\'ematiques de Toulouse, France}
\maketitle

\begin{abstract}
We discuss the Douglas--Rachford algorithm to solve the feasibility
problem for two closed sets $A,B$ in $\mathbb R^d$. We prove its local convergence to a fixed point
when $A,B$ are finite unions of convex sets. We also show that for more general 
nonconvex sets the scheme may fail to converge and start to cycle, and may then even fail to
solve the feasibility problem.

\vspace*{.2cm}
\end{abstract}

\section{Introduction}
The Douglas--Rachford iterative scheme, originally introduced in \cite{douglas}
to solve nonlinear heat flow problems, aims  
to find a point $x^*$ in the intersection of two closed constraint sets $A,B$
in $\mathbb R^d$ or in Hilbert space. 
Using monotone operator theory,
Lions and Mercier \cite{lions}  showed
that the scheme converges  weakly for two closed convex sets $A,B$ in Hilbert space with non-empty
intersection. 
A rather comprehensive analysis of  the convex case is given in \cite{combettes}.

Due to its success in applications, the Douglas--Rachford
scheme is  frequently  used in the nonconvex setting despite the lack of a satisfactory
convergence theory. 
Recently Hesse and Luke \cite{hesse}
made progress by proving local convergence of the scheme 
for   $B$ an 
affine subspace intersecting the set $A$
transversally, where $A$ is no longer convex, but satisfies
a regularity hypothesis called superregularity.  
Numerical experiments in the nonconvex case
indicate, however, that the Douglas--Rachford scheme should converge
in much more general situations.
Very frequently one observes that the iterates
settle for convergence after a  chaotic transitory phase;
see \cite{aragon,veit} and the references therein.
Here we prove local convergence of the
Douglas--Rachford scheme when $A,B$
are finite unions of convex sets. Our result is complementary to \cite{hesse},
because no transversality hypothesis is required.
This result is proved in section \ref{convergence}.

We will also  show
that for nonconvex sets $A,B$  the Douglas--Rachford scheme may fail to
converge and start to  cycle without solving the feasibility problem. We show that
this may even lead to continuous limiting cycles. These are more delicate to
construct, because in that case
the Douglas--Rachford sequence  $x_{n+1}\in T(x_n)$
is bounded and satisfies $x_{n+1}-x_n\to 0$, but fails to converge. 
Our construction is given in section \ref{failure}.

\section{Preparation}
Given a closed subset $A$ of $\mathbb R^d$, the projection
onto $A$ is the set-valued mapping $P_A$
defined as 
\[
P_A(x) = \big\{ a\in A: \|x-a\| = {\rm dist}(x,A) \big\},
\]
where $\|\cdot\|$ is the Euclidean norm, and
dist$(x,A)=\min\{\|x-a\|: a\in A\}$. The reflection of $x$ in $A$
is then the set-valued operator
\[
R_A = 2P_A - I. 
\]
Given two closed sets $A,B$ in $\mathbb R^d$, the Douglas--Rachford iterative scheme, starting at $x_0$,
generates a sequence $x_n$ by the recursion
\[
x_{n+1} \in T(x_n), \qquad T :=\textstyle \frac{1}{2}\left( R_BR_A+I \right). 
\]
We call
$T$ the Douglas--Rachford operator, or shortly,  DR operator.
Suppose $x^+\in T(x)$ is one step of the Douglas--Rachford scheme. Then $x^+$ is obtained  as
\[
x^+ =  x+b-a, 
\]
where $a \in P_A(x)$, $y=2a-x$, and
$b\in P_B(y)$.    We call  $a$   the shadow of  iterate $x$
on $A$, $b$ the reflected shadow of  $x$ in $B$, both used to produce $x^+$. We write $x^+=T(x)$ if the DR-operator is single-valued, and similarly for projectors $P_A,P_B$
and reflectors $R_A,R_B$.

The fixed point set of $T$ is defined as $\Fix(T)=\{x\in\mathbb R^d: x\in T(x)\}$.
Note that if $x^*\in \Fix(T)$, and if $a^*\in A$, $b^*\in B$ are the shadow 
and reflected shadow of $x^*$
used to produce $x^*\in T(x^*)$, then $a^*=b^*\in A \cap B$,
so every fixed point
gives  rise to a solution $a^*\in A \cap B$
of the feasibility
problem. However, in the set-valued case, 
it may happen that $x^*\in \Fix(T)$ has other shadow-reflected shadow pairs
$(\tilde{a},\tilde{b})$ leading away from
$x^*$, i.e., where $\tilde{a} \not = \tilde{b}$, so that $\tilde{x}=x^*+\tilde{b}-\tilde{a}\in T(x^*)\setminus \{x^*\}$. 
We therefore introduce the set of strong fixed-points
of $T$ as
\[
\sFix(T) = \big\{x\in \mathbb R^n: T(x)=\{x\}\big\}.
\]
Note that $A \cap B \subset \sFix(T) \subset \Fix(T)$. 
If $T$ is single-valued, then $\Fix(T)=\sFix(T)$.

These concepts are linked to discrete dynamical system theory,
where fixed points are   steady states. We recall that a steady state $x^*$
is stable in the sense of Lyapunov if for every $\epsilon > 0$ there exists $\delta > 0$
such that every trajectory $x_{n+1}\in T(x_n)$ with starting point
$x_0\in B(x^*,\delta)$ satisfies $x_n\in B(x^*,\epsilon)$ for all $n$. 
 It is clear that
$x^*\in \Fix(T)\setminus \sFix(T)$ can never be stable, because
$x_0=x^*$ produces trajectories going away from $x^*$.

\section{Unions of convex sets}
\label{convergence}
In this section
$A=\bigcup_{i\in I}A_i$
and $B=\bigcup_{j\in J} B_j$ are finite unions of closed convex sets, a case which is
of interest in a number of practical applications like
rank or sparsity optimization \cite{hesse2}, or even road design \cite{road}, 
where finite  unions of linear or affine subspaces 
are used. 
For every $i\in I$ and $j\in J$ let $T_{ij}$ be the Douglas--Rachford operator
associated with the sets $A_i,B_j$. By convexity of $A_i,B_j$, the operators $T_{ij}$ are single-valued,
and $T(x)\subset \{T_{ij}(x): i\in I, j\in J\}$.

Since $A,B$ are finite unions of convex sets, every DR step  is realized as the DR step
of one of the operators $T_{ij}$, and in that case, we say that this operator is active.
To make this precise, we define the set 
of active indices at $x$ as
\begin{eqnarray}
\label{active}
K(x) =\big\{(i,j)\in I\times J: P_{A_i}(x)\in P_A(x), P_{B_j}\left(R_{A_i}\left( x\right)\right)\in 
P_B\left( R_{A_i}(x) \right)\big\}.
\end{eqnarray}
Note that if $(i,j)\in K(x)$, then $T_{ij}(x)\in T(x)$. Conversely, for every $x^+\in T(x)$, 
there exists $(i,j)\in K(x)$ such that $x^+=T_{ij}(x)$.  However, be aware that $T_{ij}(x)\in T(x)$
may be true without $(i,j)$ being active at $x$.

\begin{theorem}
\label{theorem1}
{\bf (Local attractor).}
Let $A=\bigcup_{i\in I}A_i$ and $B=\bigcup_{j\in J}B_j$ be finite unions of closed convex sets, and
let $x^*\in \sFix(T)$ be a strong fixed point. Then $x^*$ has a radius of attraction $R>0$ 
with the following property: For every $0 < \epsilon < R$, suppose
a Douglas--Rachford trajectory $x_{n+1}\in T(x_n)$ 
enters the ball $B(x^*,\epsilon)$.  Then it
stays there 
and converges to some
fixed point $\bar{x}\in \Fix(T)$. Moreover, 
every accumulation point of the  shadow sequence $a_n\in P_A(x_n)$  
is a solution of the feasibility problem. The radius of attraction
can be computed as
\begin{eqnarray}
\label{radius}
R = \sup\big\{ \epsilon > 0: K\left(
B(x^*,\epsilon)\right)\subset K(x^*)\big\}.
\end{eqnarray}
\end{theorem}

\begin{proof}
1)
The fact that $x^*\in \sFix(T)$ is a strong fixed point has the following consequence. Whenever
$(i,j)\in K(x^*)$ is active, then $P_{A_i}\left( x^*\right) = P_{B_j}\left( R_{A_i}(x^*)\right)\in A \cap B$. 
Therefore, for every $(i,j)\in K(x^*)$, the operator
$T_{ij}$ has $x^*$ as a fixed point. 

2) We now show the following.
There exists $\epsilon>0$ such that every
$x\in B(x^*,\epsilon)$ has $K(x)\subset K(x^*)$.

Let us consider the set
$I(x)=\{i\in I: \mbox{there exists } j\in J  \mbox{ such that } (i,j)\in K(x)\}$  of active
indices $i\in  I$ at $x$. Then by definition
\begin{eqnarray}
\label{neu}
\delta_1: = \min\big\{{\rm dist}(x^*,A_i): i\not\in I(x^*)\big\} -
{\rm dist}(x^*,A)>0.
\end{eqnarray}
Similarly, we have
\begin{eqnarray}
\label{veryneu} \quad
\delta_2 := \min \big\{{\rm dist}\left(R_{A_i}(x^*),B_j
\right)-{\rm dist}\left(R_{A_i}(x^*),B\right) : i\in I(x^*),
(i,j)\not\in K(x^*)\big\}>0.
\end{eqnarray}
Choose  $\epsilon > 0$ such that $2\epsilon < \min\{\delta_1,\delta_2\}$. We show
that $\epsilon$ is as claimed.  Pick $(i,j)\not\in K(x^*)$. We have to
show that $(i,j)\not  \in K(x)$.
\
First consider the case where $i\in I\setminus I(x^*)$. We  show that $i\in I\setminus I(x)$. Indeed,
\begin{align*}
{\rm dist}(x,A) &\leq \|x-x^*\| + {\rm dist}(x^*,A) \\
&\leq \epsilon + {\rm dist}(x^*,A_i)-\delta_1  \qquad\qquad \mbox{(using (\ref{neu}))}\\
&\leq 2\epsilon + {\rm dist}(x,A_i)-\delta_1 < {\rm dist}(x,A_i),
\end{align*}
showing that $i\not \in I(x)$. 
\
We now discuss the case
where $i\in I(x^*)$, but $(i,j)\not\in K(x^*)$. That means 
$ {\rm dist}(R_{A_i}(x^*),B_j) - {\rm dist}(R_{A_i}(x^*),B)\geq \delta_2$.
Therefore
\begin{align*}
{\rm dist}(R_{A_i}(x),B) &\leq  \|R_{A_i}(x)-R_{A_i}(x^*)\| + {\rm dist}(R_{A_i}(x^*),B) \\
&\leq \epsilon + {\rm dist}(R_{A_i}(x^*),B_j)- \delta_2\qquad\qquad\qquad  \mbox{(using (\ref{veryneu}))} \\
&\leq 2\epsilon + {\rm dist}(R_{A_i}(x),B_j) - \delta_2 < {\rm dist}(R_{A_i}(x),B_j),
\end{align*}
proving $(i,j)\not\in K(x)$.

3) As an immediate consequence of 2) we have the following: If 
$x\in B(x^*,\epsilon)$ and 
$x^+\in T(x)$ is realized as $x^+ = T_{ij}(x)$ for some active operator $T_{ij}$, that is, for some
$(i,j)\in K(x)$, then
this operator $T_{ij}$ has $x^*$ as a fixed point. Namely,
by 2) $x$ satisfies $K(x)\subset K(x^*)$, hence $(i,j)\in K(x^*)$, and
therefore $P_{A_i}(x^*) = P_{B_j}(R_{A_i}(x^*))$, which proves what we claimed.

5) We next show that as soon as a DR sequence $x_{n+1}\in T(x_n)$
enters the ball $B(x^*,\epsilon)$, then it stays there and converges.

This can be seen as follows. Suppose the trajectory enters $B(x^*,\epsilon)$
at stage $n$. Then
the active operator $T_{i_nj_n}$ 
used to produce $x_{n+1}=T_{i_nj_n}(x_n)\in T(x_n)$
has $x^*$ as a fixed point, because $(i_n,j_n)\in K(x_n)\subset K(x^*)$. Therefore,
by \cite[Prop.~4.21]{bauschke_book}, this operator satisfies
$\|x_{n+1}-x^*\| = \|T_{i_nj_n}(x_n)- x^*\| \leq \|x_n-x^*\|\leq \epsilon$. The conclusion is
that from index $n$ onward the sequence $x_n$ stays in the ball $B(x^*,\epsilon)$,
and all operators $T_{i_mj_m}$ used from here on have $x^*$ as a common fixed point. 

Now we invoke Elsner {\em et al.} \cite[Thm.~1]{elser},  who show that 
$x_n$ converges to a common fixed point $\bar{x}$ of the operators
$T_{i_mj_m}$, $m\geq n$. But $\bar{x}$ is then also a fixed point of $T$, as follows
from the continuity of the distance functions. One has $\bar{x}\in B(x^*,\epsilon)$,
and moreover, if $a_n = P_{i_n}(x_n)\in A_{i_n}\subset A$, then every accumulation
point of the sequence $a_n$ is a solution of the feasibility problem. Namely,
if we consider $b_n = P_{B_j}\left( R_{A_i}(x_n)\right)\in B$, and if we take
accumulation points $a^*$ of $a_n$ and $b^*$ of $b_n$, then $a^*\in P_A(x^*)$,
$b^*\in P_B\left( 2a-x^*)\right)$, hence $a^*=b^*$, because $x^*$ is a strong fixed point.

6) To conclude let us now define the radius of attraction $R$ as in formula
(\ref{radius}). In 1) -- 5) above we have shown that there exists
$\epsilon>0$ such that $K\left( B(x^*,\epsilon) \right) \subset K(x^*)$, and that 
this inclusion alone already
implies convergence of every trajectory entering $B(x^*,\epsilon)$. This 
means that the supremum in (\ref{radius}) is over a nonempty set, and that
is all that we need. 
\end{proof}

\begin{remark}
{\bf (Stable steady state.)}
The dynamic system interpretation of Theorem \ref{theorem1} is that
a strong fixed point $x^*\in \sFix(T)$ is a stable steady state of the Douglas--Rachford dynamical system
$x^+\in T(x)$ when $A,B$
are unions of convex sets. 
Note also that we
do not claim that $\bar{x}\in \Fix(T)$ is strong, nor do we claim that the iterates
converge to $x^*$ itself.  
\end{remark}

The following
observation is also of the essence.

\begin{remark}
{\bf (Strong fixed point needed).}
Theorem \ref{theorem1} is not true if $x^*\in \Fix(T)\setminus 
\sFix(T)$, that is,
if $x^*$ is not a strong fixed point. Indeed, let
$A = \{-1,1\}$ and $B=\{-2,1\}$. Then $x^*=0$ is a fixed-point of $T$, but not a strong one. 
Now there exist arbitrarily small $\epsilon \in (0,1)$ such that trajectories starting in
$B(0,\epsilon)=(-\epsilon,\epsilon)$  will not stay in that ball. Indeed,
for $-\epsilon < x < 0$, $x^+$ will move away from $0$ and will not stay in $B(0,\epsilon)$, while for
$0<x <\epsilon$, $x^+$   stays. \hfill $\square$
\end{remark}

\begin{remark}
Note that if $A,B$ are convex sets,
then all $K(x)$ are identical singleton sets, so formula
(\ref{radius}) gives
$R=\infty$, which means 
the DR scheme converges globally. 
\end{remark}

\begin{remark}
Formula (\ref{radius})  allows to compute the radius of
attraction of
a strong fixed-point $x^*\in \sFix(T)$ in certain cases.  
For illustration, consider 
$A = \{(x,0):x\in \mathbb R\} \cup \{(0,y): y\in \mathbb R\}$  a union of two lines
and $B= \{(x,y): y = -\frac{y^*}{x^*}x + y^*\}$ a line, where $x^*>0$,
$y^*>0$. Then  $(0,y^*)$
and $(x^*,0)$ are the two only fixed points of $T$, both strong,  and one easily
finds $R(x^*,0)=x^*/\sqrt{2}$ and $R(0,y^*)=y^*/\sqrt{2}$. 
\end{remark}

\begin{remark}
{\bf (Asymptotic stability).}
Let $x^*\in \sFix(T)$ be a strong fixed point, and suppose
there exists $\delta > 0$ such that  $B(x^*,\delta )$ containing no further fixed point of $T$.
Then it follows from Theorem \ref{theorem1}
that we can find $0 < \epsilon \leq \delta$ such that every trajectory
$x_{n+1}\in T(x_n)$ entering $B(x^*,\epsilon)$ converges to $x^*$.
In the dynamical system terminology, $x^*$ is then asymptotically stable
in the sense of Lyapunov. Note that this still fails for an
isolated fixed point $x^*\in \Fix(T)\setminus
\sFix(T)$.
\end{remark}

\begin{theorem}
\label{theorem2}
{\bf (Local convergence).}
Let $A=\bigcup_{i\in I}A_i$ and $B=\bigcup_{j\in J}B_j$ be finite unions of convex sets.
Let $x_{n+1}\in T(x_n)$ be a bounded Douglas--Rachford sequence satisfying
$x_{n+1}-x_n\to 0$. Then $x_n$ converges to a fixed-point
$\bar{x}\in \Fix(T)$. Moreover, every accumulation point of the shadow sequence
$a_n\in P_A(x_n)$ is a solution to
the feasibility problem.
\end{theorem}

\begin{proof}
1)
For every $n\in \mathbb N$ let us choose an active index pair
$(i_n,j_n)\in K(x_n)$ such that $x_{n+1}=T_{i_nj_n}(x_n)$. Put
$a_n=P_{A_{i_n}}(x_n)$ and $b_n=P_{B_{j_n}}\left( R_{A_{i_n}}(x_n) \right)=P_{B_{j_n}}(2a_n-x_n)$,
so that $x_{n+1}=x_n+b_n-a_n$. Note that
$a_n-b_n = x_n-x_{n+1}\to 0$ by hypothesis. 

2) Let $x^*$ be any accumulation point of the sequence $x_n$. We define a subset
$K_0(x^*)$ of the active set $K(x^*)$ as
\[
K_0(x^*) = \big\{(i,j)\in K(x^*): P_{A_i}(x^*)=P_{B_j}\left(
R_{A_i}(x^*) \right)\big\}.
\] 
Note that every $T_{ij}$ with $(i,j)\in K_0(x^*)$ has 
$x^*$ as a fixed point. 

3) We now claim that
for every accumulation point $x^*$ of the sequence $x_n$ there exists
$\epsilon > 0$ and an index $n_0$ such that for every $x_n$ with $n\geq n_0$
and $x_n\in B(x^*,\epsilon)$, we have $(i_n,j_n)\in K_0(x^*)$. 

To prove this, assume on  the contrary that for every
$\epsilon = \frac{1}{k}$ there exists $n_k$ such that $x_{n_k} \in B(x^*,\frac{1}{k})$, but
$(i_{n_k},j_{n_k})\not \in K_0(x^*)$. Moreover, let $n_k < n_{k+1}\to \infty$. 
Then $x_{n_k}\to x^*$. Passing to another subsequence
which we also denote by $x_{n_k}$,
we may assume that $i_{n_k}=i$, $j_{n_k}=j$. Then
$a_{n_k}=P_{A_i}(x_{n_k})\to a^*\in A$,
$b_{n_k}=P_{B_j}\left( R_{A_i}(x_{n_k}) \right) \to b^*\in B$. Since $a_n-b_n\to 0$
by part 1), we deduce that $a^*=b^*\in  A \cap B$. Since we also have
$a_{n_k}=P_{A_i}(x_{n_k})\in P_A(x_{n_k})$ and $b_{n_k}\in P_B\left( R_{A_i}(x_{n_k})\right)
\subset P_B\left( R_A(x_{n_k})\right)$, we 
get $a^*\in P_A(x^*)$ and $b^*\in P_B\left( R_A(x^*)\right)$, hence
 $(i,j)\in K(x^*)$.
Since $a^*=b^*$, we have
$(i,j)\in K_0(x^*)$. This contradiction proves the claim.

4) Since $x^*$ is an accumulation point of the sequence $x_n$, there exist
infinitely many indices with $x_n\in B(x^*,\epsilon)$. Choose one with $n\geq n_0$.
Then $x_{n+1}=T_{i_n j_n}(x_n)$, and by part 3) we have
$(i_n,j_n)\in K_0(x^*)$. By part 2), $T_{i_nj_n}$ has therefore $x^*$
as a fixed point. Using \cite[Prop.~4.21]{bauschke_book}, we deduce
$\|x_{n+1}-x^*\| = \|T_{i_nj_n}(x_n)-x^*\| \leq \|x_n-x^*\|\leq \epsilon$, hence
$x_{n+1}$ stays in the ball $B(x^*,\epsilon)$. This means we can repeat the argument, 
showing that the entire sequence $x_m$, $m\geq n$, stays in $B(x^*,\epsilon)$.
By part 3) the operators $T_{i_mj_m}$, $m\geq n$, have the common fixed point
$x^*$, hence we conclude again using \cite[Thm.~1]{elser} that $x_m$
converges to some $\bar{x}$, which must be a fixed point of $T$. The second part of
the statement follows now from $a_n-b_n\to 0$.
\end{proof}

\begin{remark}
\label{example1}
{\bf (Discrete limit cycle).}
Let $A= \{(x,y): y=0\}\subset \mathbb R^2$ be the $x$-axis, fix $1\geq \eta > 0$, and
put $B = \{(0,0), (7+\eta,\eta),(7,-\eta)\}$. When started at $x_1=(7,\eta)$, 
the method cycles between the four points $x_1$, $x_2=(7+\eta,0)$, $x_3=(7+\eta,-\eta)$, $x_4=(7,0)$. 
Note
that $B$ is a finite union of bounded convex sets and $A$ is convex,  the iterates 
$x_2$ and $x_4$ reach $A$, but the method fails to converge, and it also fails to solve the feasibility
problem.
\end{remark}

\begin{remark}
{\bf (Several shadows)}.
Let $B$ be a circle in $\mathbb R^2$, and let $A$ consist in the union of two
circles which touch  $B$ from outside in $a_1^*,a_2^*$. Then the centre $x^*$ of $B$
is a fixed-point of the Douglas--Rachford operator $T$, and the two
points $a_i^*\in A \cap B$ are both shadows of $x^*$. This shows that even in the case of convergence of
$x_n$ we do not expect the shadows $a_n$ to converge. 
\end{remark}

\begin{remark}
{\bf (More than two sets.)}
It is a standard procedure in applications to extend the Douglas--Rachford
scheme to solve the feasibility
problem for a  finite number of constraint sets $C_1,\dots,C_m$ in $\mathbb R^d$. One defines
$A$ to be the diagonal in $\mathbb R^d \times \dots \times \mathbb R^d$ ($m$ times), and chooses
as $B=C_1 \times \dots \times C_m$ in the product space. Then if
$\bigcap_{i=1}^m C_i \not= \emptyset$, the Douglas--Rachford algorithm in product space
can be used to compute a point in this intersection. The interesting observation is that
if each $C_i$ is a finite union of convex sets, then this remains true for the
set $B$, hence our convergence theory applies. 
\end{remark}

\section{Existence of a continuous limit cycle}
\label{failure}
In this section we construct two closed bounded  sets
$A,B$ such that the Douglas--Rachford iteration $x_{n+1}\in T(x_n)$ with
$T=\frac{1}{2}\left( R_BR_A +I\right)$  fails to converge 
and produces a continuum of accumulation points $F \subset \Fix(T)$ forming
a continuous limit cycle.
We let $A$ be the cylinder mantle
\[
A = \big\{ (\cos t,\sin t,h): 0 \leq t \leq 2\pi, 0 \leq h \leq
1\big\},
\]
and $B$ a double 
spiral consisting of two logarithmic spirals in 3D winding down against the cylinder,
one from inside, one from outside.
That is, 
\[
B= \big\{ ((1\pm e^{-t})\cos t,(1 \pm e^{-t} )\sin t, e^{-t/2}):
0 \leq t\big\} \cup F,
\]
where $F = \{(\cos\alpha,\sin\alpha,0): \alpha \in [0,2\pi]\}$. Note that 
$A \cap B=F$. We will construct a Douglas--Rachford sequence $x_{n+1}=T(x_n)$, whose set
of accumulation points is the entire set $F$.  It will be useful to divide
the spiral in its outer and inner part
\[
B_\pm = \big\{ ((1\pm e^{-t})\cos t,(1 \pm e^{-t} )\sin t,
e^{-t/2}): 0 \leq t\big\}  \cup F
\]
so that $B= B_-\cup B_+$ and $B_-\cap B_+=F$ with these notations. 

\begin{theorem}
\label{theorem3}
{\bf (Continuous limit cycle).}
Let $x_{n+1} = T(x_n)$ be any Douglas--Rachford sequence between $A$
and $B$ with starting point $x_1\in B_-\setminus F$. Then the sequence $x_n$ is bounded,
satisfies, $x_{n+1}-x_n\to 0$, but fails to converge. Its set of accumulation points
is $F=A\cap B \subset \sFix(T)$. 
\end{theorem}

\begin{proof}
1)
For $t\geq 0$ 
let us introduce the notations
\[
a(t)= \big(\cos t,\sin t, e^{-t/2}\big) \in A,
\]
and 
\[
b_{\pm}(t) = \big((1\pm e^{-t})\cos t,(1\pm e^{-t})\sin
t,e^{-t/2}\big)\in B_{\pm}\setminus F.
\]
The set $\{a(t): t\geq 0\}\subset A$ is the shadow of the spiral
on the cylinder mantle. Namely, 
it is clear that for $t > 0$,
\begin{eqnarray}
\label{fifth}
P_A\left( b_+(t) \right) = P_A\left( b_-(t)\right) = a(t).
\end{eqnarray}
In particular, 
\begin{eqnarray}
\label{fourth}
\|b_\pm(t) - P_A\left( b_\pm(t) \right)\| = e^{-t}.
\end{eqnarray} 
In consequence
\[
R_A\left( b_+(t)\right) = b_-(t), \quad R_A\left( b_-(t)\right) = b_+(t).
\]
In words, 
the two branches $B_\pm$ of the double spiral are the reflections of each others
in the cylinder mantle. 

2)
Let us now analyze
the projection of $a(t)$ on the double spiral $B$. We consider the projections
of $a(t)$ on each of the branches $B_\pm$. 
We start with the analysis of $b\in P_{B_+}\left( a(t) \right)$. We first claim
that $b\not\in F$. Indeed, the point $v\in F$ closest to
$a(t)$ is $v= (\cos t,\sin t,0)=P_F\left( a(t)\right)$, so $\|v-a(t)\| = e^{-t/2}$. 
But $\|b_+(t) - a(t)\| = e^{-t} < e^{-t/2}$, hence there are points
on $B_+\setminus F$ closer to $a(t)$ than $v$. This shows that any projected point
$b\in P_{B_+}(a(t))$ has to be of the form $b_+(\tau)$ for
some $\tau \geq 0$. Now consider some such $b_+(\tau)\in P_{B_+}\left( a(t)\right)$, then
\begin{eqnarray}
\label{first}
e^{-t} = \| a(t)-b_+(t) \| \geq \| a(t) - b_+(\tau) \| \geq \| a(\tau) - b_+(\tau)\| = e^{-\tau},
\end{eqnarray}
which shows $\tau \geq t$. Here the second estimate follows from $a(\tau)=P_A\left( b_+(\tau)\right)$.

3)
Let us further observe that $\tau > t$. Namely, if we had
$\tau = t$, then we would have a fixed point pair for the method of alternating projections between
$A$ and $B_+$ in the sense that $a(t) = P_A\left( b_+(t)\right)$, $b_+(t)\in P_{B_+}\left(a(t) \right)$.
That would mean the distance squared
$\tau \mapsto \frac{1}{2} \| a(t) - b_+(\tau)\|^2$ had a local minimum
at $\tau = t$. But the derivative of this function at $\tau = t$ is $-e^{-2t} < 0$, so
$\tau = t$ is impossible, and we deduce $\tau > t$.

4)
Using $b_+(\tau)\in P_{B_+}\left( a(t)\right)$ and (\ref{first}), we find
\[
e^{-t}\geq \|a(t)-b_+(\tau)\| \geq | e^{-t/2}-e^{-\tau/2}| = e^{-t/2}-e^{-\tau/2}
=e^{-t} \left( e^{t/2}-e^{t-\tau/2} \right), 
\]
which shows 
\[
0 < e^{t/2} - e^{t-\tau/2} \leq 1.
\]
This can be re-arranged as 
\begin{eqnarray}
\label{bound}
0 < 1 - e^{t/2-\tau/2} \leq e^{-t/2}.
\end{eqnarray}
In particular, for $t\to \infty$ we must have $\tau-t\to 0$. 

5) 
Let us next show that the projection $b_+(\tau) = P_{B_+}\left( a(t)\right)$ is unique
for $t$ sufficiently large. Indeed, suppose we find $t<\tau_1<\tau_2$
such that $b_+(\tau_1),b_+(\tau_2)\in P_{B_+}\left( a(t)\right)$. Then by 
(\ref{bound}) we have $t < \tau_1 < \tau_2 < t - 2 \log\left( 1-e^{-t/2} \right)$.
Define the function $d_+(x) = \frac{1}{2} \|a(t)-b_+(x)\|^2$, then
$d_+(t) = \frac{1}{2} e^{-2t}$, $d_+'(t) = -e^{-2t} < 0$. Since
$\tau_1,\tau_2$ are local minima, we have
$d_+'(\tau_1)=d_+'(\tau_2)=0$. But
$d_+''(x) =  \cos(t-x) + 2e^{-2x}+2e^{-x}\sin(t-x)+\frac{3}{2}e^{-x}-\frac{1}{4}e^{-x/2-t/2}$. 
 In consequence, for $t$ large and $x$ moving in the interval
$x\in \left(t,t - 2 \log\left( 1-e^{-t/2} \right)\right)$,   we have $d_+''(x) \approx \cos(t-x)\approx 1$,
so certainly $d_+''(x)>0$ for these $x$, and since the local minima
$\tau_i$ are in that interval for $t$ large,
$d_+'(\tau_2)=0$ is impossible. 
This proves $b_+(\tau)=P_{B_+}\left( a(t)\right)$ for $t$ sufficiently large.

\centerline{
\includegraphics[width=16cm,height=8cm]{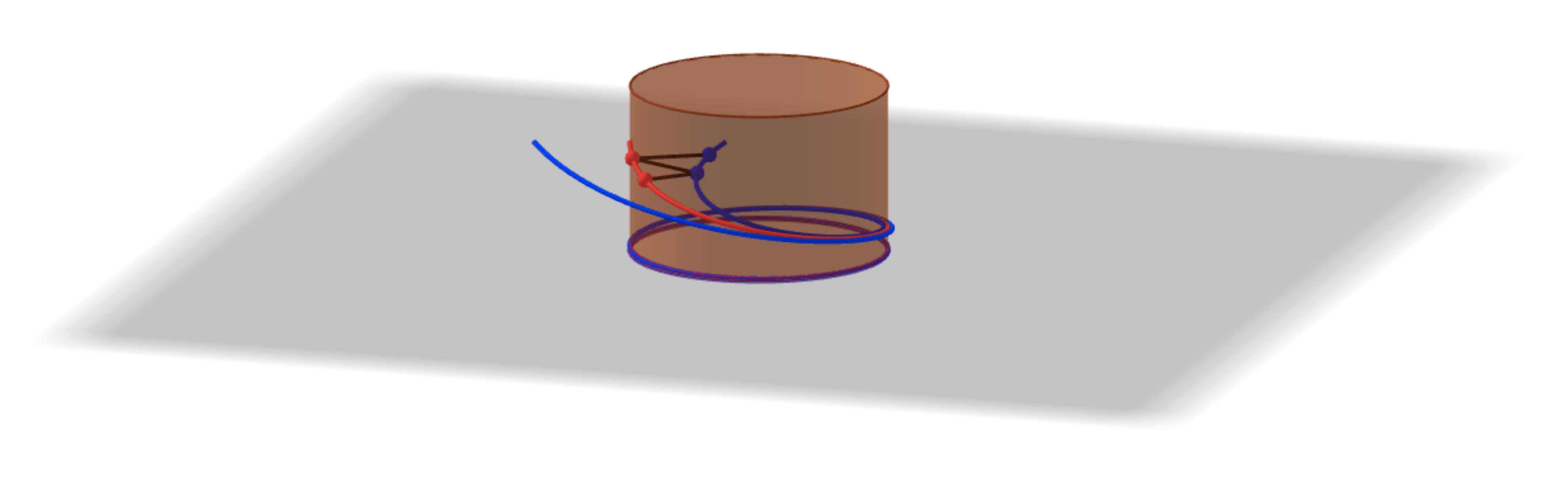}
}

\hspace*{-.3cm}
6)
Let us now consider the point $b_-(\tau)\in B_-\setminus F$ on the inner spiral. 
We claim that
$b_-(\tau)$ is closer to $a(t)$ than $b_+(\tau)$. 
Indeed, the set
of points $w$ having equal distance to $b_-(\tau)$ and $b_+(\tau)$ is the
tangent plane to the cylinder at the point $a(\tau)=\frac{1}{2}\left( b_-(\tau)+b_+(\tau) \right)$. But the cylinder lies
in one of the half-spaces associated with this plane, namely the one containing $b_-(\tau)$,
 $\|b_-(\tau) - a(t)\| < \|b_+(\tau)-a(t)\|$. Since $b_+(\tau)$ is the nearest point to $a(t)$
in $B_+$, we deduce that $P_B\left( a(t) \right) \subset P_{B_-}\left( a(t)\right)$
for all $t$. 
In other words, projections from the shadow of the spiral onto the double spiral always go 
to the inner spiral.

We could also use an analytic argument to prove this. Let $d_-(x)=\frac{1}{2}\|a(t)-b_-(x)\|^2$
and consider the function
$f(x) =d_+(x)-d_-(x)$, then $f(x)=\frac{1}{2}e^{-x}\left( 1-\cos(x-t) \right)$, so $f\ge 0$,
and $f=0$ for $x=t+2k\pi$.
Since $t<\tau < t-2\log(1-e^{-t/2})\ll t + 2\pi$, this proves $f(\tau) > 0$.

7) Let $b\in P_B\left( a(t)\right) = P_{B_-}\left( a(t)\right)$ a projected point of $a(t)$
in the inner spiral. We know already that $b\not\in F$, hence
$b = b_-(\sigma)$ for some $\sigma \geq 0$. Repeating the argument
in part 2), it follows that $\sigma > t$.  Indeed, like in (\ref{first}) we have
\[
e^{-t} = \|a(t)-b_-(t) \| \geq \| a(t)-b_-(\sigma) \| \geq \|a(\sigma)-b_-(\sigma)\| = e^{-\sigma},
\]
and the same argument as in part 2) shows $\sigma > t$.
But then again
\[
e^{-t} \geq \|a(t) - b_-(\sigma) \| \geq | e^{-t/2}-e^{-\sigma/2}| = e^{-t/2}-e^{-\sigma/2}
= e^{-t} \left(  e^{t/2}-e^{t-\sigma/2} \right),
\]
which shows
\[
0 < e^{t/2} - e^{t-\sigma/2} \leq 1.
\]
This can be re-arranged to
\begin{eqnarray}
\label{second}
0 < 1 - e^{t/2-\sigma/2} \leq e^{-t/2}.
\end{eqnarray}
In particular, for $t\to \infty$ we must have $\sigma -t\to 0$,  and in particular $0<\sigma-t \ll 2\pi$.
Therefore projected
points $b_-(\sigma)\in P_B\left( a(t)\right)$ lie
on the same tour of the spiral as $a(t),b_-(t)$,  and one does not  take 
shortcuts by jumping down a full turn of the spiral $B_-$ or more. Repeating the argument of 5),
we also see that $b_-(\sigma)=P_{B_-}\left( a(t)\right)$ is unique
for $t>0$ large enough.

8)
Let us now generate our Douglas--Rachford sequence $x_n$, starting at
$x_1=b_-(t_1)\in B_-$ with $t_1 > 0$, excluding $t_1=0$ for simplicity to have
a unique projection on the cylinder mantle at the start. 

We get $a_1=P_A(x_1)= a(t_1)$, hence  $R_A(x_1)= b_+(t_1)$. Since
$b_+(t_1)\in B$, it is its own reflection in $B$, and we get $b_1 = b_+(t_1)$. Averaging then
gives $x_2 = (b_1+x_1) /2= a_1 = a(t_1)$, which concludes the first DR-step.

The second DR-step proceeds as follows. Since $x_2 = a(t_1)\in A$, it is its own reflection in $A$,
so $a_2 = x_2 = a_1=R_A(x_1)$. Now let  $b_2=P_B(R_A(x_2))$, then
$b_2 = P_B\left( a(t_1) \right) = P_{B_-}\left( a(t_1) \right)$, so $b_2 = b_-(t_2)$ for some
$t_2 > t_1$, where $t_2$ is for $t_1$ what
$\sigma$ was for $t$ in part 7). So the reflected point
is $2b_-(t_2)-a(t_2)$ and averaging then gives
$x_3 = b_-(t_2)\in B_-$. 

Proceeding in this way, we generate a strictly increasing sequence $t_n$ such that
\begin{eqnarray}
\label{dr}
x_{2k-1} = b_-(t_k) \in B_-, \quad x_{2k} = a(t_k)\in A.
\end{eqnarray}
Moreover, the shadow and reflected shadow are
\[
P_A\left(x_{2k-1} \right) = a_{2k-1}=a(t_k), \quad P_B\left( R_A\left( x_{2k-1}\right)\right)=b_{2k-1}=b_+(t_k)
\]
respectively,
\[
P_A\left( x_{2k} \right)= a_{2k}=a(t_k),\quad P_B\left( R_A\left( x_{2k}\right)\right)= b_{2k}=b_-\left( t_k\right).
\]
Furthermore, note that
we generate a sequence of alternating projections
between $A$ and $B_-$. Namely
\begin{eqnarray}
\label{map}
b_0:= x_1 \stackrel{P_A}{\rightarrow} a_1=a_2 \stackrel{P_{B_-}}{\rightarrow}{ b_2}
\stackrel{P_A}{\rightarrow} a_3=a_4 \stackrel{P_{B_-}}{\rightarrow}b_4 \dots
\end{eqnarray}

9) We now argue that $t_n$ so constructed tends to $\infty$. Suppose on the contrary that
$t_n < t_{n+1} \to t^* < \infty$. Then from the construction we see that
we create a pair $a(t^*)\in A$, $b_-(t^*)\in B_-$ such that
$a(t^*)\in P_{A}(b_-(t^*))$ and $b_-(t^*)\in P_{B_-}(a(t^*))$. Arguing as before,
this would imply that $\tau \mapsto \frac{1}{2} \| a(t^*)-b_-(\tau)\|^2$ had a local
minimum at $\tau = t^*$, which it does not because its derivative  at $t^*$ is $-e^{-2t^*} < 0$. Hence
$t^* < \infty$ is impossble, and we have $t_n\to \infty$. As a consequence,
the statements about uniqueness of the operators and the estimate (\ref{first})
are now satisfied from some counter $n_0$ onward.

10)
To conclude, observe that $x_n-x_{n+1} \to 0$ by (\ref{second}), and that the $a(t_n)$ are $2e^{-t_n}$-dense
in the interval $[t_n,t_n+2\pi]$, because of
\[
\|a(t_n)-a(t_{n+1})\| \leq \|a(t_n)-b_-(t_{n+1})\| + \|b_-(t_{n+1})-a(t_{n+1})\| \leq e^{-t_n} + e^{-t_{n+1}}
\leq 2e^{-t_n}.
\] 
Using (\ref{fourth}) and the fact that every $a(t)$ with $t\in [t_n,t_n+2\pi]$ is at distance
$\leq e^{-t_n/2}$ to the set $F$, we deduce that every point in $F$
is an accumulation point of both the DR-sequence (\ref{dr}) and the MAP sequence (\ref{map}).
\end{proof}

\begin{remark}
{\bf (Strong fixed points need not be stable).}
The system-theoretic interpretation of this result
is that a strong fixed-point $x^*\in F\subset \sFix(T)$ need not be a stable steady state. This
is in contrast with Theorem \ref{theorem1}, where this was shown to
be true when $A,B$ are finite unions of convex sets. A second interpretation is that
$F$ is a stable attractor for the dynamic system $x^+=T(x)$. 
\end{remark}

\begin{remark}
{\bf (Shadows need not converge).}
We note that not only does the DR sequence $x_n$
fail to converge in Theorem \ref{theorem3}, also the sequences
$a_n=P_A(x_n)\in A$, $b_n=P_B(R_A(x_n))\in B$ fail to converge
and have  the same continuum set of accumulation points $F$.
Presently no example of failure of convergence of a local sequence $x_n$ is known
where the shadow sequence $a_n$ converges to a single limit $a^*\in A \cap B$.
It is clear that this could only happen when $F \subset \{x: \|x-a^*\|=\epsilon\}$
for some $\epsilon > 0$.
\end{remark}

\begin{corollary}
{\bf (Continuous limit cycle for MAP).}
Let $x_n$ be the Douglas--Rachford sequence constructed above,
and let $a_n\in A$, $b_n\in B$ be the shadows associated with $x_n$. Then
$b_{2n}, a_{2n+1}$ is a sequence of alternating projections between the 
cylinder $A$ and the inner spiral $B_-$. This sequence also fails
to converge and has the same set of accumulation points $F=A\cap B$.
\hfill$\square$
\end{corollary}

\begin{corollary}
{\bf (Limit cycle for MAP with one set convex).}
Every sequence of alternating projections $a_n,b_n$ between
the outer spiral $B_+$ and the solid cylinder {\rm conv}$(A)$
started at $b_1\in B_+\setminus F$ is bounded, satisfies $a_n-b_n\to 0$,
but fails to converge and has the set $F = B_+ \cap {\rm conv}(A)$
as its set of accumulation points. 
\end{corollary}

\begin{proof}
In part 2) of the proof of Theorem \ref{theorem3} we  analyzed this sequence,
which is generated by the building blocks $a(t) \to b_+(\tau) \to a(\tau)$. 
\end{proof}

\begin{remark}
Here we have an example of a semi-algebraic
convex set
conv$(A)$, and the spiral $B_+$, which is the projection of a semi-analytic set in $\mathbb R^4$,
where the MAP sequence fails to converge and leads to a continuous limit cycle. The first example
with a continuous limit cycle appears in \cite{spiral}, but with more pathological sets $A,B$. 
The fact that $B_+$ is not subanalytic can be deduced
from   \cite[Cor.~7]{aude}. 
Currently we do not have an example where the DR-algorithm
fails to converge and creates a continuous limit cycle with one of the sets
convex.
\end{remark} 

\section*{Acknowledgements}
\noindent
H.H.\ Bauschke  was supported by the Natural Sciences and Engineering Research Council of Canada and the Canada Research Chair Program. D. Noll acknowledges hospitality of the University of British 
Columbia Kelowna and support by the Pacific Institute of the Mathematical Sciences
during the preparation of this
paper.
The visualization was made possible  by GeoGebra \cite{geogebra}.

\end{document}